\documentclass[11pt,a4paper]{amsart}
\usepackage[all]{xy}
\usepackage{amssymb,amsmath}
\usepackage{fourier}
\usepackage[in]{fullpage}

\usepackage{color}
\definecolor{Chocolat}{rgb}{0.36, 0.2, 0.09}
\definecolor{BleuTresFonce}{rgb}{0.215, 0.215, 0.36}
\definecolor{EgyptianBlue}{rgb}{0.06, 0.2, 0.65}

\usepackage[colorlinks,final,hyperindex]{hyperref}
\hypersetup{citecolor=BleuTresFonce, urlcolor=EgyptianBlue, linkcolor=Chocolat}

\newtheorem{theorem}{Theorem}[section]
\newtheorem*{itheorem}{Theorem}
\newtheorem{corollary}[theorem]{Corollary}

\newtheorem{proposition}[theorem]{Proposition}

\theoremstyle{definition}

\newtheorem{definition}[theorem]{Definition}

\newcommand{\ac}{\scriptstyle \text{\rm !`}}

\DeclareMathOperator{\dgAlg}{\mathsf{wdg-alg}}
\DeclareMathOperator{\dgOp}{\mathsf{dg-op}}

\DeclareMathAlphabet{\pazocal}{OMS}{zplm}{m}{n}

\def\calA{\pazocal{A}}
\def\calB{\pazocal{B}}
\def\calC{\pazocal{C}}

\def\calG{\pazocal{G}}

\def\calI{\pazocal{I}}

\def\calP{\pazocal{P}}
\def\calQ{\pazocal{Q}}

\def\calT{\pazocal{T}}
\def\calU{\pazocal{U}}

\def\calX{\pazocal{X}}

\DeclareMathOperator{\Hom}{Hom}

\DeclareMathOperator{\id}{id}

\DeclareMathAlphabet{\mathbbold}{U}{bbold}{m}{n}

\def\k{\mathbbold{k}}

\begin{document}

\title[A Quillen adjunction between algebras and operads]{A Quillen adjunction between algebras and operads,\\ Koszul duality, and the Lagrange inversion formula}

\author{Vladimir Dotsenko}
\address{School of Mathematics, Trinity College, Dublin 2, Ireland}
\email{vdots@maths.tcd.ie}

\dedicatory{To Andr\'e Joyal, who always asks right questions}

\begin{abstract}
We define, for a somewhat standard forgetful functor from nonsymmetric operads to weight graded associative algebras, two functorial ``enveloping operad'' functors, the right inverse and the left adjoint of the forgetful functor. Those functors turn out to be related by operadic Koszul duality, and that relationship can be utilised to provide examples showing limitations of two standard tools of the Koszul duality theory. We also apply these functors to get a homotopical algebra proof of the Lagrange inversion formula. 
\end{abstract}

\maketitle

\tableofcontents

\section*{Introduction}

It is well known (and utilised in many contexts) that the direct sum of components of any operad $\calP$ has a pre-Lie algebra structure, where the pre-Lie product is given, for $\alpha\in\calP(n)$, $\beta\in\calP(m)$, by the formula
 \[
\alpha\triangleleft\beta=\sum_{i=1}^n\alpha\circ_i\beta .
 \]
If the operad $\calP$ is nonsymmetric, there is an even simpler formula that has good algebraic properties:
 \[
\alpha\cdot\beta =\alpha\circ_1\beta 
 \]
(the reason for requiring the operad to be nonsymmetric is that for a symmetric operad there is no canonical way to choose one of the slots of an operation). By a direct inspection, this operation is associative, and the standard ``arity minus one'' weight grading is respected by this operation, leading to a weight graded algebra $\mathsf{A}(\calP)$. 

\smallskip 

In this paper, we define, for a weight graded associative algebra $\calA$, two different ways to ``revert'' this construction, producing, in a functorial way, two ``enveloping'' nonsymmetric operads which we denote~$\calU_{\min}(\calA)$ and~$\calU_{\max}(\calA)$. Each of these constructions has its own merit. We establish (Proposition \ref{prop:adjunction}) that the functor~$\calU_{\min}$ is the right inverse of the abovementioned functor $\mathsf{A}$, and, quite notably, the functor~$\calU_{\max}$ is the left adjoint of the functor $\mathsf{A}$, the adjunction being a Quillen adjunction for the standard choices of model structures on the corresponding categories. 

\smallskip 

There is also a somewhat unexpected relationship between the functors~$\calU_{\min}$ and~$\calU_{\max}$ coming from the operadic Koszul duality. More precisely, we establish the following result (where the superscript ${}^{\ac}$ denotes the Koszul dual cooperad, and $\calU^c_{\min}$ and $\calU^c_{\max}$ denote the analogous constructions for cooperads). 
\begin{itheorem}[{Corollary \ref{cor:MinMaxKoszulDuality}}] Suppose that a weight graded associative algebra $\calA$ is quadratic. 
\begin{itemize}
\item[(i)] Both operads\, $\calU_{\min}(\calA)$ and\, $\calU_{\max}(\calA)$ are quadratic as well.
\item[(ii)] For the corresponding homogeneous quadratic presentations of\, $\calU_{\min}(\calA)$ and\, $\calU_{\max}(\calA)$, we have 
\begin{gather*}
(\calU_{\min}(\calA))^{\ac}=\calU^c_{\max}(\calA^{\ac}) ,\\
(\calU_{\max}(\calA))^{\ac}=\calU^c_{\min}(\calA^{\ac}) .
\end{gather*}
\item[(iii)] The operads\, $\calU_{\min}(\calA)$ and\, $\calU_{\max}(\calA)$ are Koszul if and only if the algebra $\calA$ is Koszul. 
\end{itemize}
\end{itheorem}
Let us remark that for the left adjoint of the functor producing a pre-Lie algebra from an operad (whether symmetric, nonsymmetric or shuffle), a similar result is not available: the corresponding left adjoint is quite ill-behaved with respect to the Koszul duality, as one can see already for the rather trivial example of the one-dimensional pre-Lie algebra with zero product. 

\smallskip 

The result we just stated allows us to convert various examples relevant in Koszul duality theory for weight graded associative algebras into respective examples for operads. Koszul duality for operads has two frequently used tests of Koszulness, and we exhibit examples showing that both of these tests are inconclusive. 

If one suspects that the given operad $\calP$ is not Koszul, a standard way to confirm that relies on the Ginzburg--Kapranov functional equation \cite{GiKa}. That equation, in modern terms, says that for a Koszul operad $\calP$, we have
 \[
g_{\calP}(g_{\calP^{\ac}}(t))=t ,
 \]
where $g$ is the Poincar\'e series of the operad (the generating series for the Euler characteristics of components). Thus, if one is lucky to know both series $g_{\calP}(t)$ and $g_{\calP^{\ac}}(t)$, one can check directly that they are not inverse to each other in order to prove that $\calP$ is not Koszul. This was utilised, for instance, in \cite{GoRe1,GoRe2} for operads of Lie-admissible and power-associative algebras associated to certain subgroups of $S_3$, in \cite{Dzh} for the operad of Novikov algebras and in \cite{DzhZu} for the operad of alternative algebras. In \cite{DzhZu}, it was asked whether there exists a non-Koszul operad
$\calP$ for which that functional equation holds. In this paper, we answer that question positively, establishing (Theorem \ref{th:GKCounter}) that there exist a Koszul operad $\calP_1$ and a non-Koszul operad $\calP_2$ for which 
 \[
g_{\calP_1}(t)=g_{\calP_2}(t)\quad  \text{ and  }\quad  g_{\calP_1^{\ac}}(t)=g_{\calP_2^{\ac}}(t) . 
 \]
In fact, for the operads $\calP_1$ and $\calP_2$ one can take the results of applying the functor $\calU_{\max}$ to the algebras invented by Piontkovski \cite{Piont} who showed that the Koszulness test for weight graded algebras based on the Backelin's functional equation \cite{Ba} is inconclusive. It makes sense to remark here that there exists a completely different example of non-Koszul weight graded algebras defying the Ginzburg--Kapranov test invented by Positselski \cite{Posic} who used Manin products of quadratic algebras to ``truncate'' algebras, removing components of large weight grading. In general, Manin products for quadratic operads are ill-suited for Koszul duality purposes~\cite{Va}, although it is possible that applying ideas similar to those of \cite[Prop.~7]{Do} leads to well-behaved truncations that allow one to mimic the approach of~\cite{Posic}.

The other way round, if one suspects that the given operad $\calP$ is Koszul, a convenient way to confirm that relies on operadic Gr\"obner bases \cite{DoKh}. More precisely, it is known that an operad with a quadratic Gr\"obner basis of relations is Koszul. Thus, if one can guess a monomial order for which the ideal of relations of $\calP$ has a quadratic Gr\"obner basis, this automatically proves that $\calP$ is Koszul. On several occasions the author of this paper was asked for examples of a Koszul operad whose Koszulness cannot be established using this criterion.  In this paper, we establish (Theorem \ref{th:GrobnerCounter}) that there exists a Koszul operad for which there is no choice of a monomial order such that its ideal of relations has a quadratic Gr\"obner basis. Once again, an example can be obtained applying the functor $\calU_{\max}$ to the example of Berger \cite{Berger} defying the corresponding Koszulness test for weight graded algebras.

\smallskip 

A completely different application of enveloping operads is that to combinatorics. We are able to demonstrate (Theorem \ref{th:LIT}) that the bar / cobar duality for operads paired with an elementary combinatorial lemma of Raney \cite{Raney} leads to a new proof of the Lagrange inversion formula for power series. There are interpretations of Lagrange inversion from the category theory viewpoint, e.g. \cite{ErMe,GeLa,Ha,Jo}, but to the best of our knowledge, they are quite different from the proof we present.

\subsection*{Acknowledgements} A substantial part of this paper was written at the Max Planck Institute for Mathematics in Bonn, where the author's stay was supported through the programme ``Higher Structures in Geometry and Physics''; the author thanks MPIM for for the excellent working conditions enjoyed during his visit. 
The author is also grateful to Martin Markl and Elisabeth Remm for useful discussions, and to Imma Galvez-Carillo and Andy Tonks for their hospitality during the final stage of working on the paper. Special thanks are to Andr\'e Joyal who asked the author about the possible application to the Lagrange inversion formula.  

\section{Background, notations, and recollections}\label{sec:recollections}

We recall here some basic background information; for a general treatment of operads, we refer the reader to \cite{LoVa}, for a comprehensive coverage of  Gr\"obner bases to \cite{BrDo}, for an extensive treatment of model categories to \cite{Ho}. All vector spaces and (co)chain complexes throughout this paper are defined over an arbitrary field~$\k$. To handle suspensions of chain complexes, we introduce an element~$s$ of degree~$1$, and define, for a graded vector space~$V$, its suspension $sV$ as $\k s\otimes V$. 

\subsection{Weight graded algebras and nonsymmetric operads}

The key point of this paper is unravelling a relationship between homotopical properties of \emph{weight graded algebras} and \emph{nonsymmetric operads}. Both of those are monoids in the category of nonsymmetric collections, the former for the \emph{tensor product of collections} 
 \[
(\calP\otimes\calQ)(n)=\bigoplus_{i+j=n}\calP(i)\otimes\calQ(j) ,
 \]
and the latter for the \emph{(nonsymmetric) composition of collections}
 \[ 
(\calP\circ\calQ)(n)=\bigoplus_{k\ge 0}\calP(k)\otimes\calQ^{\otimes k}(n) .
 \]
For brevity, we use the abbreviation `ns' instead of the word `nonsymmetric'.  We assume all weight graded algebras $\calP$ \emph{connected} ($\calP(0)\cong\k$). We also assume all ns operads \emph{reduced} ($\mathcal{P}(0)=0$) and \emph{connected} ($\mathcal{P}(1)\cong\k$). We use the notation $\mathcal{X}\cong\mathcal{Y}$ 
for isomorphisms of ns collections, and the notation $\mathcal{X}\simeq\mathcal{Y}$ for weak equivalences (quasi-isomorphisms). 

We insist on the terminology ``weight graded'' since we shall use the word ``graded'' on its own in the sense which is standard for homotopy algebra purposes, referring to homological grading and its associated Koszul sign rule. Moreover, we shall mainly consider the categories $\dgAlg$ and $\dgOp$ which are, respectively, the categories of weight graded algebras and ns operads whose components are chain complexes, and morphisms are chain maps respecting the algebra / operad structure.

We recall that ns operads can also be defined in terms of \emph{infinitesimal (partial) composition products} $\circ_i$; those products on an operad $\calP$ must satisfy the 
\emph{parallel} and \emph{sequential} axioms: 
\begin{gather*}
(f_1\circ_j f_2)\circ_i f_3=(-1)^{|f_2||f_3|}(f_1\circ_i f_3)\circ_{j+n_3-1} f_2,\\
f_1\circ_i(f_2\circ_k f_3)=(f_1\circ_i f_2)\circ_{k-i+1}f_3 
\end{gather*}
 for all $f_1\in\calP(n_1)$, $f_2\in\calP(n_2)$, $f_3\in\calP(n_3)$, and $1\le i<j\le n_1$, $1\le k\le n_2$.

The free operad generated by a ns collection $\mathcal{X}$ is denoted $\mathcal{T}(\mathcal{X})$, the cofree (conilpotent) cooperad cogenerated by a ns collection $\mathcal{X}$ is denoted $\mathcal{T}^c(\mathcal{X})$; the former is spanned by ``tree tensors'' (planar trees where each internal vertex $v$ carries a label from $\calX(k)$, where $k$ is the number of inputs of $v$), and has its \emph{composition product}, and the latter has the same underlying ns collection but a different structure, a \emph{decomposition coproduct}. The abovementioned ns collection is weight graded (a tree tensor has weight $p$ if its underlying tree has $p$ internal vertices), and we denote by $\mathcal{T}(\mathcal{X})^{(p)}$ the subcollection which is the span of all elements of weight $p$.

\subsection{Koszul duality for quadratic (co)operads}

A pair consisting of a ns collection $\mathcal{X}$ and a subcollection $\mathcal{R}\subset\mathcal{T}(\mathcal{X})^{(2)}$ is called \emph{quadratic data}. To a choice of quadratic data one can associate the \emph{quadratic operad $\calP=\calP(\mathcal{X},\mathcal{R})$ with generators $\mathcal{X}$ and relations $\mathcal{R}$}, the largest quotient operad $\mathcal{O}$ of $\mathcal{T}(\mathcal{X})$ for which the composite
\[
\mathcal{R}\hookrightarrow\mathcal{T}(\mathcal{X})^{(2)}\hookrightarrow\mathcal{T}(\mathcal{X})\twoheadrightarrow\mathcal{O}
\]
is zero. Also, to a choice of quadratic data one can associate the \emph{quadratic cooperad $\mathcal{C}=\mathcal{C}(\mathcal{X},\mathcal{R})$ with cogenerators $\mathcal{X}$ and corelations $\mathcal{R}$}, the largest subcooperad $\mathcal{Q}\subset\mathcal{T}^c(\mathcal{X})$ for which the composite
 \[
\mathcal{Q}\hookrightarrow\mathcal{T}^c(\mathcal{X})\twoheadrightarrow\mathcal{T}^c(\mathcal{X})^{(2)}\twoheadrightarrow\mathcal{T}^c(\mathcal{X})^{(2)}/\mathcal{R}
 \]
is zero.
The Koszul duality for operads assigns to every quadratic operad $\calP=\calP(\mathcal{X},\mathcal{R})$ its \emph{Koszul dual cooperad}
$\calP^{\ac} := \mathcal{C}(s\mathcal{X},s^2\mathcal{R})$.   
Recall that the (left) Koszul complex of a ns quadratic operad $\calP=\calP(\mathcal{X},\mathcal{R})$ is the ns collection $\calP\circ\calP^{\ac}$ equipped with a certain differential coming from a ``twisting morphism''
 \[
\varkappa\colon \mathcal{C}(s\mathcal{X},s^2\mathcal{R})\twoheadrightarrow s\mathcal{X}\to \mathcal{X}\hookrightarrow\calP(\mathcal{X},\mathcal{R}) , 
 \]
see \cite[Sec.~7.4]{LoVa} for details. This allows to define Koszul operads: a quadratic operad $\calP$ is said to be \emph{Koszul} if its Koszul complex is acyclic, so that the inclusion
 \[
\k\cong(\calP\circ\calP^{\ac})(1)\hookrightarrow\calP\circ\calP^{\ac} 
 \]
induces an isomorphism in the homology. 

The bar complex $\mathsf{B}(\calP)$ is, by definition, the cofree cooperad $\mathcal{T}^c(s\calP)$ equipped with the differential coming from the infinitesimal products $\circ_i$ on $\calP$. This differential graded cooperad is bi-graded (by the syzygy degree and the weight degree), and it is known \cite[Prop.~7.3.2]{LoVa} that for a quadratic operad $\calP$ the ``diagonal'' part (the weight degree exceeds the syzygy degree by exactly one) of its homology is isomorphic to $\calP^{\ac}$. According to \cite[Th.~7.4.6]{LoVa}, a quadratic operad $\calP$ is Koszul if and only if the homology of its bar complex is concentrated on the diagonal. 

Definitions and results recalled in this section are also available for weight graded algebras, see, e.g. \cite[Ch.~3]{LoVa}; the corresponding bar complexes are cofree (conilpotent) coalgebras. 

\subsection{Poincar\'e series for algebras and operads}

A very useful numerical invariant of a ns collection is given by its Poincar\'e series. For a ns collection $\calX$ with finite-dimensional components, its  \emph{Poincar\'e series} $g_{\mathcal{X}}(t)$ is the generating series for Euler characteristics of components of $\calX$: 
 \[
g_{\mathcal{X}}(t) = \sum_{n\ge 0}\chi(\mathcal{X}(n))t^n .
 \]
An important property of the Poincar\'e series is that it is compatible with both tensor products and nonsymmetric compositions:
Let $\mathcal{X}$ and $\mathcal{Y}$ be two ns collections  with finite-dimensional components. Then
\begin{gather*}
g_{\mathcal{X}\otimes\mathcal{Y}}(t)=g_{\mathcal{X}}(t)g_{\mathcal{Y}}(t) ,\\
g_{\mathcal{X}\circ\mathcal{Y}}(t)=g_{\mathcal{X}}(g_{\mathcal{Y}}(t)) .
\end{gather*}
Inspecting the definitions of Koszul algebras and operads, one arrives at the \emph{Backelin functional equation} \cite{Ba} for connected weight graded Koszul algebras 
 \[
g_{\calA}(t)g_{\calA^{\ac}}(t)=1 
 \]
and the \emph{Ginzburg--Kapranov functional equation} \cite{GiKa} for reduced connected ns Koszul operads 
 \[
g_{\calP}(g_{\calP^{\ac}}(t))=t .
 \]

\subsection{Gr\"obner bases for ns operads}

The free ns operad $\calT(\calX)$ has a basis of \emph{tree monomials}: if we choose a linear basis for each component $\calX(n)$, a tree monomial is tree tensor whose vertices are labelled by the basis elements of $\calX$. Each composition of tree monomials is again a tree monomial. 

There exist several ways to introduce a \emph{monomial order}: a total ordering of tree monomials in such a way that the operadic compositions are compatible with that  ordering. Once a  The naive combinatorial definition of \emph{divisibility} of tree monomials agrees with the operadic definition: one tree monomial occurs as a subtree in another one if and only if the latter can be obtained from the former by operadic compositions. Once a monomial order is fixed, one can perform a ``long division'' modulo any collection $\calG$ of elements of $\calT(\calX)$, producing for every element its \emph{reduced form}. A \emph{Gr\"obner basis} of an ideal $\calI$ of the free operad $\calT(\calX)$ is a system $\calG$ of generators of~$\calI$ for which the leading monomial of every element of $\calI$ is divisible by one of the leading terms of elements of~$\calG$. One can show that the cosets of the tree monomials that are not divisible by the leading terms of the Gr\"obner basis form a basis of the quotient $\calT(\calX)/\calI$, so each element has a unique reduced form modulo a Gr\"obner basis. The most powerful Gr\"obner basis criterion is the \emph{Diamond Lemma}, which states that $\calG$ is a Gr\"obner basis if and only if for each common multiple of two leading monomials of $\calG$, a certain element (\emph{S-polynomial} associated to that common multiple) has a reduced form equal to zero.

\subsection{Model categories of weight graded algebras and ns operads}

A model category is a category $\mathsf{C}$ together with three distinguished classes of maps: \emph{weak equivalences}, \emph{fibrations} and \emph{cofibrations}; these maps have to satisfy certain axioms which we do not list here. According to the philosophy of Quillen \cite{Qu}, this is precisely the kind of data needed to describe the \emph{homotopy category} $\mathop{\mathsf{Ho}}(\mathsf{C})$, which is the localisation of $\mathsf{C}$ with respect to the weak equivalences. It follows from the result of Hinich \cite{Hi} that each of the categories $\dgAlg$ and $\dgOp$ admits a model category structure for which 
\begin{itemize}
\item the class $\mathfrak{W}$ of weak equivalences is given by the quasi-isomorphisms (morphisms that induce isomorphisms on the homology);
\item the class $\mathfrak{F}$ of fibrations is given by the morphisms which are component-wise epimorphisms; 
\item the class $\mathfrak{C}$ of cofibrations is given by the morphisms which satisfy the left lifting property with respect
to acyclic fibrations $\mathfrak{F}\cap\mathfrak{W}$.
\end{itemize}
We shall refer to those model category structures as \emph{standard} model category structures on $\dgAlg$ and $\dgOp$.

\smallskip 

For two model categories $\mathsf{C}$ and $\mathsf{D}$, a pair of adjoint functors 
 \[
L\colon \mathsf{C}\rightleftharpoons\mathsf{D} \colon R
 \]
implements a \emph{Quillen adjunction} if the following equivalent conditions are satisfied:
\begin{itemize}
\item[(i)] $L$ preserves cofibrations and acyclic cofibrations;
\item[(ii)] $R$ preserves fibrations and acyclic fibrations.
\end{itemize}
Quillen adjunctions are of importance since by Quillen's total derived functor theorem they induce adjunctions of the respective homotopy categories. 

\section{Two kinds of enveloping ns operads and their functorial properties}

In this section, we describe a forgetful functor from the category of dg ns operads to the category of weight graded dg algebras and two kinds of ``enveloping operads''; one of them is the left adjoint of the forgetful functor, while the other is its right inverse. 

\smallskip 

For a ns collection $\calA$, we define its shifts up $\calA^+$ and down $\calA^-$ by the respective formulas
 \[
\calA^+(n)=
\begin{cases}
\qquad 0, \qquad n=0,\\
\calA(n-1), \,\,\,\, n\ge 1, 
\end{cases}
\qquad\qquad 
\calA^-(n)=
\calA(n+1) ,\,\,\, n\ge 0.
 \]

\begin{definition}\label{def:OpToAlg}
The functor 
 \[
\mathsf{A}\colon \dgOp\to \dgAlg
 \]
is defined as follows. On the level of ns collections, we have
 \[
\mathsf{A}(\calP) = \calP^- .
 \]
The product 
 \[
\mathsf{A}(\calP)(k)\otimes\mathsf{A}(\calP)(l)=\calP(k+1)\otimes\calP(l+1)\to \calP(k+l+1)= \mathsf{A}(\calP)(k+l), 
 \]
is defined using the infinitesimal composition at the first slot $\circ_1$ of the operad $\calP$. Associativity of the product follows from the sequential axiom of ns operads.
\end{definition}

The protagonists of this paper are the following two kinds of enveloping operads.

\begin{definition}\label{def:AlgToOp} Let $\calA$ be a weight graded dg algebra. 
\begin{enumerate}
\item The \emph{max-envelope operad}~$\calU_{\max}(\calA)$ is defined as the quotient of the free ns operad $\calT(\calA^+)$ by the ideal generated by the element $1-\id$ and all elements of the form
 \[
a_1\circ_1 a_2-a_1\cdot a_2\ . 
 \]
\item The \emph{min-envelope operad}~$\calU_{\min}(\calA)$ is defined as the quotient of the free ns operad $\calT(\calA^+)$ by the ideal generated by the element $1-\id$ and all elements of the form
 \[
a_1\circ_1 a_2-a_1\cdot a_2 , \qquad a_1\circ_i a_2 \quad (i\ge 2) \ .
 \]
\end{enumerate}
In both cases, we take $a_1\in \calA(k-1)=\calA^+(k)$ and $a_2\in \calA(l-1)=\calA^+(l)$ for some $k$ and $l$, so that $a_1\cdot a_2\in\calA(k+l-2)$, which we identify with $\calA^+(k+l-1)$; we also use the identification $\calA(0)=\calA^+(1)$ when writing the element $1-\id$ .
\end{definition}

We begin with describing the underlying ns collections of~$\calU_{\min}(\calA)$ and~$\calU_{\max}(\calA)$ explicitly. For that, we shall use Gr\"obner bases for ns operads. In order to do that, it is necessary to choose a linear basis of generators, and a monomial order of the basis of tree monomials in the free ns operad arising from that choice of a basis.  In our case, choosing a basis of $\calA$ leads to a natural choice of a basis of $\calT(\calA^+)$ consisting of tree monomials where each internal vertex with $k$ inputs is decorated by a basis element of $\calA(k+1)$. As a monomial order, we use the path degree lexicographic order.

\begin{theorem}\label{th:EnvelopeExplicit} Let $\calA$ be a weight graded algebra. 
\begin{itemize}
\item[(i)] On the level of underlying ns collections, we have
\begin{equation}\label{eq:Umin}
\calU_{\min}(\calA)\cong \calA^+  .
\end{equation}
\item[(ii)] The underlying ns collection of~$\calU_{\max}(\calA)$ has a basis formed by the cosets of tree monomials for which the leftmost input of each internal vertex is a leaf.
\end{itemize}
\end{theorem}

\begin{proof}
The relation $1-\id$ can be used to eliminate the generator $1$ completely, which we shall do right away; from here onwards, we consider the presentations of our operads where the set of generators is identified with the augmentation ideal of $\calA$. 
 
Let us first demonstrate that for both operads~$\calU_{\min}(\calA)$ and~$\calU_{\max}(\calA)$, their defining relations form a Gr\"obner basis. For the latter operad, all the S-polynomials corresponding to overlaps of the leading terms can be reduced to zero due to associativity of the product in $\calA$. For the former one, there are three types overlaps of the leading terms of defining relations. There are overlaps where both leading terms use the infinitesimal composition in the first slot, and for them the corresponding S-polynomials are reduced to zero due to associativity of $\calA$. There are overlaps where exactly one of the leading terms uses the infinitesimal composition in the first slot, and the corresponding S-polynomial is immediately reduced to zero using the defining relations of the type $a_1\circ_i a_2$ $(i\ge 2)$. Finally, there are overlaps where neither of the leading terms uses the infinitesimal composition in the first slot, and the corresponding S-polynomial is equal to zero. Overall, we see that the Diamond Lemma applies, and the defining relations form a Gr\"obner basis in both cases.

The explicit description of the underlying ns collections of the operads~$\calU_{\min}(\calA)$ and~$\calU_{\max}(\calA)$ easily follows from the fact that the leading terms of the Gr\"obner basis we found are precisely all the tree monomials with two internal vertices in the case of~$\calU_{\min}(A)$ and all the tree monomials with two internal vertices that are obtained by a composition in the first slot in the case of~$\calU_{\max}(\calA)$. 
\end{proof}

The first application of Theorem \ref{th:EnvelopeExplicit} is to describing the functorial properties of our two constructions. Note that both enveloping operads we defined make sense for weight graded dg algebras, giving rise to functors from the category $\dgAlg$ to the category $\dgOp$. 

\begin{proposition}\label{prop:adjunction}\leavevmode
\begin{enumerate}
\item The functor~$\calU_{\min}\colon\dgAlg\to\dgOp$ is the right inverse of the functor $\mathsf{A}\colon \dgOp\to\dgAlg$.
\item The functors 
 \[
\calU_{\max} \colon \dgAlg \rightleftharpoons \dgOp \colon \mathsf{A}
 \]
form a pair of adjoint functors, where~$\calU_{\max}$ is left adjoint to $\mathsf{A}$. Moreover, this adjunction is a Quillen adjunction for the standard model category structures.
\end{enumerate}
\end{proposition}

\begin{proof}
To establish the first statement, we need to check that for every weight graded dg algebra $\calA$, we have 
 \[
\mathsf{A}(\calU_{\min}(\calA))\cong \calA ,
 \]
where the isomorphism is natural in $\calA$. By Formula \eqref{eq:Umin}, the underlying ns collection of~$\calU_{\min}(\calA)$ is~$\calA^+$, so the underlying ns collection of $\mathsf{A}(\calU_{\min}(\calA))$ is $(\calA^+)^-\cong \calA$. The product in this algebra is clearly the same as the original product in $\calA$.

To prove the second statement, we need to show that for every weight graded dg algebra $\calA$ and every dg operad $\calP$, we have
 \[
\Hom_{\dgOp}(\calU_{\max}(\calA),\calP) \cong\Hom_{\dgAlg}(\calA,\mathsf{A}(\calP)) ,
 \]
the isomorphism being natural in $\calA$ and $\calP$. Let us note that since the underlying ns collection of $\mathsf{A}(\calP)$ is $\calP^-$, a map $f\in \Hom_{\dgAlg}(\calA,\mathsf{A}(\calP))$ is the same as a collection of maps $f_n\colon \calA(n)\to\calP^-(n)$, $n\ge0$, which are maps of chain complexes satisfying for all $p,q\ge 0$ the constraints
 \[
f_p(b_1)\circ_1 f_q(b_2)=f_{p+q}(b_1\cdot b_2) .
 \]
Also, since~$\calU_{\max}(\calA)$ is generated by $\calA^+$ modulo the defining relations $a_1\circ_1 a_2-a_1\cdot a_2$, a map $g\in \Hom_{\dgOp}(\calU_{\max}(\calA),\calP)$ is the same as a collection of maps of chain complexes
$g_n\colon\calA^+(n)\to\calP(n)$, $n\ge 1$, satisfying for all $k, l\ge 1$ the constraints
 \[
g_k(b_1)\circ_1 g_l(b_2)-g_{k+l-1}(b_1\cdot b_2)=0 . 
 \] 
Now it is obvious that the assignment $g_n:=f_{n-1}$ establishes a natural one-to-one correspondence between the hom-sets we consider. The Quillen adjunction claim follows from the fact that the functor $\mathsf{A}$ only shifts the arity by one without changing the individual components of the underlying collections, so it clearly preserves fibrations and acyclic fibrations.
\end{proof}

\section{Enveloping operads and Koszul duality}

It turns out that the two kinds of enveloping operads that we defined are very well suited for purposes of homotopical algebra, in particular, they agree nicely with the bar constructions for weight graded algebras and ns operads.  To make this statement precise, let us remark that there are analogous constructions for cooperads. Let $\calC$ be a weight graded dg coassociative coalgebra . The coproduct in $\calC$ induces an operation 
 \[
\bigtriangleup\colon\calC^+(k+l-1)=\calC_{k+l-2}\to\bigoplus_{k,l} \calC(k-1)\otimes \calC(l-1)=\bigoplus_{k,l}  \calC^+(k)\otimes\calC^+(l) . 
 \]
The cooperad~$\calU^c_{\max}(C)$ is the subcooperad of the cofree cooperad $\mathcal{T}^c(\calC^+)$ cogenerated by $\calC^+$ with corelations  $\Delta_{1}(c)=\bigtriangleup(c)$ and $\Delta_i(c_1\circ_i c_2)=c_1\otimes c_2$, for all $c\in \calC^+(k+l-1)$, $c_1\in\calC^+(k)$, $c_2\in\calC^+(l)$, and all $2\le i\le k$. Its underlying collection admits a description analogous to the above one. The cooperad~$\calU^c_{\min}(\calC)$ is the subcooperad of the cofree cooperad $\calT^c(\calC^+)$ cogenerated by $\calC^+$ with corelations $\Delta_{1}(c)=\bigtriangleup(c)$ for all $c\in \calC^+(k+l-1)$. 

\begin{theorem}\label{th:EnvelopingBar}
Let $\calA$ be a weight graded associative algebra. There exist quasi-isomorphisms of dg ns cooperads
\begin{gather}
\mathsf{B}(\calU_{\max}(\calA))\simeq\calU^c_{\min}(\mathsf{B}(\calA)) ,\label{eq:BUmax=UBmin}\\
\mathsf{B}(\calU_{\min}(\calA))\simeq\calU^c_{\max}(\mathsf{B}(\calA)) . \label{eq:BUmin=UBmax}
\end{gather}
\end{theorem}

\begin{proof}
This result also follows from the computation of operadic Gr\"obner bases we performed when proving Theorem \ref{th:EnvelopeExplicit}, paired with the inhomogeneous Koszul duality theory \cite{GTV}. Indeed, the Gr\"obner bases that we found readily demonstrate that for both operads~$\calU_{\max}(\calA)$ and~$\calU_{\min}(\calA)$ with their defining presentations, the two conditions of the inhomogeneous Koszul duality,  that is the minimality of the space of generators $(ql_1)$ and the maximality of the space of relations $(ql_2)$, are satisfied, so both operads~$\calU_{\min}(\calA)$ and~$\calU_{\max}(\calA)$ are inhomogeneous Koszul. 

The space of relations of the associated homogeneous quadratic operad $q\calU_{\max}(\calA)$ is spanned by all elements $a_1\circ_1 a_2$, so the space of corelations of the Koszul dual cooperad $q\calU_{\max}(\calA)^{\ac}$ is spanned by  all elements $s a_1\circ_1 s a_2$. Thus, as a ns collection, that cooperad is $\mathsf{B}(\calA)^+$, and examining the differential which recalls the linear parts of relations, we see by a direct inspection that 
 \[
\calU_{\max}(\calA)^{\ac}\cong \calU^c_{\min}(\mathsf{B}(\calA)) .
 \]
Since the inhomogeneous quadratic operad~$\calU_{\max}(\calA)$ is Koszul, we have
 \[
\mathsf{B}(\calU_{\max}(\calA))\simeq\calU_{\max}(\calA)^{\ac}\cong \calU^c_{\min}(\mathsf{B}(\calA)) ,
 \]
proving the first claim of the theorem. 

The space of relations of the associated homogeneous quadratic operad $q\calU_{\min}(\calA)$ is spanned by all elements $a_1\circ_i a_2$ for all $i$ (the distinction between $i=1$ and $i\ge 2$ disappears after forgetting the linear terms), so the space of corelations of the Koszul dual cooperad $q\calU_{\max}(\calA)^{\ac}$ is spanned by  all elements $s a_1\circ_i s a_2$. Thus, if we ignore the differential of that cooperad, it is nothing but the cofree ns cooperad 
$\calT^c(s\calA^+)$. The differential recalls the linear parts of relations, so the only contributions to the differential come from collapsing the first input edge of a vertex, which leads to an isomorphism
 \[
\calU_{\min}(\calA)^{\ac}\cong \mathsf{B}(\calU_{\min}(\calA)) ,
 \]
which is somewhat clear given that our generators of~$\calU_{\min}(\calA)$ actually form a linear basis of that operad, so the inhomogeneous Koszul duality provides the bar-cobar resolution \cite{LoVa}. We shall however reinterpret that operad in a slightly different way now.  We note that each internal vertex of the underlying tree of every tree monomial in $\calT^c(s\calA^+)$ is included in a unique ``spine'', the maximal chain of internal vertices for which the edge connecting every two neighbouring vertices is the first input of one of them. To each of these spines we may assign a decomposable tensor in $\mathsf{B}(A)$, the tensor product of the labels of the vertices along the spine in the order from the root to the leaf. This identifies $\calT^c(s\calA^+)$ with the underlying ns collection of~$\calU^c_{\max}(\mathsf{B}(\calA))$. By a direct inspection of the differentials, we see that  
 \[
\mathsf{B}(\calU_{\min}(\calA))\cong \calU_{\min}(\calA)^{\ac}\cong \calU^c_{\max}(\mathsf{B}(\calA)) ,
 \]
proving the second claim of the theorem (in fact, in this case we have an isomorphism, not merely a weak equivalence).
\end{proof}

This implies the following result relevant for the Koszul duality theory.

\begin{corollary}\label{cor:MinMaxKoszulDuality} Suppose that a weight graded associative algebra $\calA$ is quadratic. 
\begin{itemize}
\item[(i)] Both operads\, $\calU_{\min}(\calA)$ and\, $\calU_{\max}(\calA)$ are quadratic as well.
\item[(ii)] For the corresponding homogeneous quadratic presentations of\, $\calU_{\min}(\calA)$ and\, $\calU_{\max}(\calA)$, we have 
\begin{gather}
(\calU_{\min}(\calA))^{\ac}=\calU^c_{\max}(\calA^{\ac}) , \label{eq:DUmin=UmaxD}\\
(\calU_{\max}(\calA))^{\ac}=\calU^c_{\min}(\calA^{\ac}) . \label{eq:DUmax=UminD}
\end{gather}
\item[(iii)] The operads\, $\calU_{\min}(\calA)$ and\, $\calU_{\max}(\calA)$ are Koszul if and only if the algebra $\calA$ is Koszul. 
\end{itemize}
\end{corollary}

\begin{proof}
It is enough to demonstrate that each of these properties can be all formulated in terms of bar complexes, since then Theorem \ref{th:EnvelopingBar} applies right away. For~(i), quadraticity means that the second homology of the (co)bar complex is concentrated in weight $2$. For~(ii), we note that if $\calA$ is a weight graded algebra and $\calP$ is a ns operad,  then $\calA^{\ac}$ is the diagonal part (weight equals homological degree) of $H_\bullet(\mathsf{B}(\calA))$ and $\calP^{\ac}$ is the diagonal part of $H_\bullet(\mathsf{B}(\calP))$. Finally, for~(iii), we recall that the  Koszul property for $\calA$ (and $\calP$) means that the homology of the bar complex is concentrated on the diagonal.  (Also, statements (i) and (ii) can be proved by a direct inspection of the definitions.)
\end{proof} 

\section{Counterexamples for Koszul duality theory}\label{sec:counterexamples}

Let us explain how our enveloping operads can be used to convert algebraic counterexamples into operadic ones. We continue working with ns operads, however if one prefers symmetric operads, it is sufficient to consider the result of the symmetrisation functor \cite{LoVa} applied to the corresponding ns operads. 

\subsection{The Ginzburg--Kapranov functional equation does not imply Koszulness}

In this section, we explain that a counterexample of Piontkovski~\cite{Piont} showing that the Backelin functional equation for quadratic associative algebras does not guarantee the Koszul property can be used to produce similar counterexamples for (say) binary quadratic operads. 

\begin{theorem}\label{th:GKCounter}
There exist a Koszul operad $\calP_1$ and a non-Koszul operad $\calP_2$ for which 
 \[
g_{\calP_1}(t)=g_{\calP_2}(t)\quad  \text{ and  }\quad  g_{\calP_1^{\ac}}(t)=g_{\calP_2^{\ac}}(t) . 
 \]
\end{theorem}

\begin{proof}
According to \cite{Piont}, there  exist a Koszul weight graded algebra $\calA_1$ and a non-Koszul weight graded algebra $\calA_2$ for which $g_{\calA_1}(t)=g_{\calA_2}(t)$ and $g_{\calA_1^{\ac}}(t)=g_{\calA_2^{\ac}}(t)$. Let us consider the operads $\calP_1=\calU_{\max}(\calA_1)$ and $\calP_2=\calU_{\max}(\calA_2)$; by Corollary~\ref{cor:MinMaxKoszulDuality}~(iii), the former is Koszul and the latter is not Koszul. By Theorem \ref{th:EnvelopeExplicit}~(ii), the underlying ns collections of $\calP_1$ and $\calP_2$ are completely determined by those of $\calA_1$ and $\calA_2$ respectively, so 
 \[
g_{\calP_1}(t)=g_{\calP_2}(t) . 
 \]
Finally, by Formula~\eqref{eq:DUmax=UminD}, we have $\calP_1^{\ac}=\calU^c_{\min}(\calA_1^{\ac})$ and $\calP_2^{\ac}=\calU^c_{\min}(\calA_2^{\ac})$, so by a cooperad analogue of Formula~\eqref{eq:Umin}, we have 
 \[
g_{\calP_1^{\ac}}(t)=g_{(\calA_1^{\ac})^+}(t)=tg_{\calA_1^{\ac}}(t)=tg_{\calA_2^{\ac}}(t)=g_{(\calA_2^{\ac})^+}(t)=g_{\calP_2^{\ac}}(t) .
 \]
\end{proof}

\subsection{The Gr\"obner basis criterion does not imply non-Koszulness}

Currently, the most general criterion allowing to prove that an operad is Koszul is the Gr\"obner basis criterion, stating that an operad with a quadratic Gr\"obner basis is Koszul. In this section, we explain that a counterexample of Berger \cite{Berger} showing that there exist Koszul algebras whose Koszulness cannot be established by exhibiting a Gr\"obner basis can be used to produce similar counterexamples for (say) binary quadratic operads. 

\begin{theorem}\label{th:GrobnerCounter}
There exists a Koszul operad $\calP$ for which there is no choice of a monomial order making the defining quadratic relations of that operad a Gr\"obner basis.
\end{theorem}

\begin{proof}
Let us denote by $\calA$ the three-dimensional Sklyanin algebra; it is presented by the generators $a,b,c$ and relations 
\begin{gather*}
ab-ba=c^2,\\
bc-cb=a^2,\\
ca-ac=b^2.
\end{gather*}
We consider the operad $\calP=\calU_{\max}(\calA)$. Sklyanin algebras are known to be Koszul \cite{TV}, so by Corollary~\ref{cor:MinMaxKoszulDuality}~(iii), the operad $\calP$ is Koszul. Let us show that there is no monomial order for which the defining relations of that operad form a Gr\"obner basis. If we denote by $\alpha, \beta, \gamma$ the generators of the operad $\calP$ corresponding to $a,b,c$, then the defining relations of that operad are
\begin{gather*}
\alpha\circ_1\beta-\beta\circ_1\alpha=\gamma\circ_1\gamma,\\
\beta\circ_1\gamma-\gamma\circ_1\beta=\alpha\circ_1\alpha,\\
\gamma\circ_1\alpha-\alpha\circ_1\gamma=\beta\circ_1\beta.
\end{gather*}
Assume that there exists a monomial order in the free ns operad $\mathcal{T}(\alpha,\beta,\gamma)$ for which these relations form a Gr\"obner basis. By Theorem \ref{th:EnvelopeExplicit}~(ii), the underlying ns collections of~$\calU_{\max}(\calA)$ is completely determined by that of $\calA$, so under our assumption the underlying ns collection of the operad $\calP$ is the same as for the operad~$\calU_{\max}(\calA')$, where $\calA'$ is the quadratic algebra whose relations are the leading monomials of the relations above. However, by \cite[Lemma 5.2]{Berger}, we have $\dim(\calA'(3))>10=\dim(\calA(3))$. Given that trivially we have
 \[
\dim(\calA'(0))=\dim(\calA(0))=1, \quad  \dim(\calA'(1))=\dim(\calA(1))=3, \quad \text{ and  }\quad \dim(\calA'(2))=\dim(\calA(2))=6 , 
 \]
Theorem \ref{th:EnvelopeExplicit}~(ii) immediately implies that $\dim(\calU_{\max}(\calA')(3))>\dim(\calU_{\max}(\calA)(3))$, a contradiction. 
\end{proof}

In fact, the arguments of \cite[Sec.~5]{Berger} can be adapted \emph{mutatis mutandis} to prove a stronger result: there is no linear change of basis in the space of generators 
for which the defining relations of the operad~$\calU_{\max}(\calA)$ above form a Gr\"obner basis. 

\section{Enveloping operads and the Lagrange inversion formula}\label{sec:Lagr}

In this section, we explain how to use our definitions to give a new interpretation of the Lagrange inversion formula. This proof uses one idea appearing in Raney's proof of Lagrange inversion \cite{Raney}, but, as far as we can see, is conceptually different. Recall that, for a  formal power series $F(t)$, the coefficient of $t^k$ in $F(t)$ is denoted by $\left[t^k\right]F(t)$, and the compositional inverse of $F(t)$, if exists, is denoted $F(t)^{\langle -1\rangle}$.  

\begin{theorem}[Lagrange's inversion formula]\label{th:LIT}
Let $f(t)$ be a formal power series without a constant term and with a nonzero coefficient of $t$.  Then $f(t)$ has a compositional inverse, and
\begin{equation}\label{eq:LIT}
\left[t^n\right]f(t)^{\langle -1\rangle}=\frac1n \left[u^{n-1}\right]\left(\frac{u}{f(u)}\right)^n .
\end{equation}
\end{theorem}

\begin{proof}
First, note that it is enough to prove Formula \eqref{eq:LIT} for power series $f(t)$ where the coefficient of $t$ equal to $1$ (using the change of variables $t\mapsto \lambda t$). Moreover, one may assume that all other coefficients of $f(t)$ are non-negative integers; this is true because each coefficient of the inverse series of $f(t)$ is  a polynomial expressions in finitely many coefficients of $f(t)$, so it is enough to check the equality at all non-negative integer points. Each such series $f(t)$ is of the form $t\,h(t)$, where $h(t)$ is a power series of a weight graded associative algebra (for instance we can consider algebras for which the product of any two homogeneous elements of positive degrees is equal to zero). 
By~\cite[Th.~2.3.2]{LoVa}, always have $\k\simeq \calA\otimes \mathsf{B}(\calA)$, which implies $g_\calA(t)g_{\mathsf{B}(\calA)}(t)=1$. Similarly, by~\cite[Th.~6.6.2]{LoVa}, we always have $\k\simeq\calP\circ\mathsf{B}(\calP)$, which implies $g_{\calP}(g_{\mathsf{B}(\calP)}(t))=t$. Thus, by Formulas~\eqref{eq:Umin} and~\eqref{eq:BUmin=UBmax}, we have
 \[
\left(f(t)\right)^{\langle -1\rangle}=
\left(t\,g_\calA(t)\right)^{\langle -1\rangle}=
g_{\calA^+}(t)^{\langle -1\rangle}=
g_{\calU_{\min}(\calA)}(t)^{\langle -1\rangle}=
g_{\mathsf{B}(\calU_{\min}(\calA))}(t)=
g_{\calU_{\max}(\mathsf{B}(\calA))}(t).
 \]

Let us show that for each weight graded algebra $\calB$, we have 
\begin{equation}\label{eq:Lagr}
\left[t^n\right] g_{\calU_{\max}(\calB)}(t)=\frac1n \left[u^{n-1}\right]g_{\calB^{\otimes n}}(u) 
\end{equation}
for all $n\ge1$. We shall use the basis of the underlying ns collection of~$\calU_{\max}(\calB)$ from Theorem \ref{th:EnvelopeExplicit}~(ii). To each basis element $\alpha$ of~$\calU_{\max}(\calA)(n)$, we associate a decomposable tensor $v_1\otimes\cdots \otimes v_n$ in $\calB^{\otimes n}$ as follows. If the leaf $i$ of $\alpha$ is the leftmost input of its parent vertex, and the label of that vertex is an element $b\in\calB$, we put $v_i=b$, else we put $v_i=1$. Since the weight grading of $\calB$ corresponds to the ``arity minus one'' weight grading on the operad, it is clear that for all $i=1,\ldots, n-1$ we have
 \[
|v_1|+\cdots+|v_i|\ge i,
 \] 
and that $|v_1|+\cdots+|v_n|=n-1$. Moreover, each decomposable tensor $v_1\otimes \cdots\otimes v_n$ with $|v_1|+\cdots+|v_i|\ge i$ for all $i=1,\ldots,n-1$ and with $|v_1|+\cdots+|v_n|=n-1$ gives rise to a basis monomial; we merely form the subsequence $v_{i_1}$, \ldots, $v_{i_k}$ of all factors of positive degrees, and let 
 \[
\alpha=(\cdots (v_{i_1}\circ_{i_2} v_{i_2})\circ_{i_3} v_{i_3}\cdots )\circ_{i_k} v_{i_k} .
 \]
By \cite[Th.~2.1]{Raney},  for any sequence of non-negative integers $k_1,\ldots,k_n$ with $k_1+\cdots+k_n=n-1$, there exists a unique cyclic shift such that $k_1+\cdots+k_i\ge i$ for all $i=1,\ldots,n-1$. Therefore, for every decomposable tensor $w_1\otimes \cdots\otimes w_n$ with $|w_1|+\cdots+|w_n|=n-1$, there exists a unique cyclic shift for each the degree conditions above are satisfied. Passing to the Poincar\'e series, this proves Formula~\eqref{eq:Lagr}. Applying that formula to $\calB=\mathsf{B}(\calA)$, we see that
 \[
\left[t^n\right] g_{\calU_{\max}(\mathsf{B}(\calA)}(t)=\frac1n \left[u^{n-1}\right]g_{\mathsf{B}(\calA)^{\otimes n}}(u)=\frac1n \left[u^{n-1}\right](g_{\mathsf{B}(\calA)}(u))^n .
 \]
Finally, we recall that 
 \[
g_{\mathsf{B}(\calA)}(u)=g_\calA(u)^{-1}=\frac{u}{f(u)},
 \]
which completes the proof.
\end{proof}

\bibliographystyle{amsplain}
\providecommand{\bysame}{\leavevmode\hbox to3em{\hrulefill}\thinspace}

\end{document}